\documentclass[11pt]{article}

\usepackage{amsmath}
\usepackage{amssymb}
\usepackage{amsthm}
\usepackage{amsfonts}
\usepackage{color}

\newcounter{AbcT}

\newtheorem {theo}    {Theorem}[section]
\newtheorem {Theorem}    {Theorem}[section]

\newtheorem {Lemma}      [Theorem]    {Lemma}
\newtheorem {Corollary}  [Theorem]    {Corollary}
\newtheorem {cor}  [Theorem]    {Corollary}

\theoremstyle{definition}   
\newtheorem {Remark}[Theorem]           {Remark}  
\newtheorem {Definition} [Theorem]{Definition}

\newcommand{\ignore}[1]{}

\def\eps{\epsilon}
\def\E{{\bf{E}}}
\def\P{{\bf{P}}}
\def\Var{{\bf{Var}}}

\def\R{{\mathbb{R}}}

\def\v0{{\bf 0}}

\def\0{\hat{0}}
\def\1{\hat{1}}

\def\path{{\tt path}}

\def\phi{\varphi}

\def\be{\begin{equation}}
\def\ee{\end{equation}}

\definecolor{Red}{rgb}{1,0,0}
\definecolor{Blue}{rgb}{0,0,1}
\definecolor{Olive}{rgb}{0.41,0.55,0.13}
\definecolor{Green}{rgb}{0,1,0}
\definecolor{MGreen}{rgb}{0,0.8,0}
\definecolor{DGreen}{rgb}{0,0.55,0}
\definecolor{Yellow}{rgb}{1,1,0}
\definecolor{Cyan}{rgb}{0,1,1}
\definecolor{Magenta}{rgb}{1,0,1}
\definecolor{Orange}{rgb}{1,.5,0}
\definecolor{Violet}{rgb}{.5,0,.5}
\definecolor{Purple}{rgb}{.75,0,.25}
\definecolor{Brown}{rgb}{.75,.5,.25}
\definecolor{Grey}{rgb}{.5,.5,.5}
\definecolor{Black}{rgb}{0,0,0}

\title{
Sharp Thresholds for Monotone Non Boolean Functions and Social Choice Theory
}

\author{Gil Kalai \thanks{Hebrew University of Jerusalem, Yale University, and Microsoft, Israel. Supported by  ISF, and NSF awards} \and Elchanan Mossel \thanks{U.C. Berkeley and Weizmann Institute of Science. Supported by DMS 0548249 (CAREER) award, by DOD ONR grant  N0014-07-1-05-06, by ISF grant 1300/08 and by a Minerva Grant}
}

\begin{document}

\maketitle

\thispagestyle{empty}

\begin{abstract}

A key fact in the theory of Boolean functions $f : \{0,1\}^n \to
\{0,1\}$ is that they often undergo sharp thresholds. For example:
if the function $f : \{0,1\}^n \to \{0,1\}$
is monotone and symmetric under a transitive action with $\E_p[f] = \eps$ and $\E_q[f] =
1-\eps$ then $q-p \to 0$ as $n \to \infty$. Here $\E_p$ denotes the
product probability measure on $\{0,1\}^n$ where each coordinate takes
the value $1$ independently with probability $p$.

The fact that symmetric functions undergo sharp thresholds is
 important in the study of random graphs and constraint satisfaction 
problems as well as in social choice.

In this paper we prove sharp thresholds for monotone functions
taking values in an arbitrary finite sets. We also provide
examples of applications of
the results to social choice and to random graph problems.

Among the applications is an analog for Condorcet's jury theorem and an
indeterminacy result for a large class of social choice functions.

\end{abstract}

\section{Introduction}
\subsection {Sharp thresholds}
A key fact in the theory of Boolean functions is that monotone
symmetric functions undergo sharp thresholds. This fact has
fundamental significance in the study of constraint satisfaction
problems, random graph processes and percolation and in social
choice.

The results of~\cite{FriedgutKalai:96} show that for every $0 < \eps <
1/2$ there exists a $C > 0$ such that for all $n$ and all
$f : \{0,1\}^n \to \{0,1\}$ which is monotone and symmetric (see
definitions below) if $\E_p[f] = \eps$ and $\E_q[f] = 1-\eps$ then
$0 < q-p < C (\log n)^{-1}$.
This result implies in particular that ``Every monotone graph property
has a sharp threshold'' as~\cite{FriedgutKalai:96} is titled.
It also implies that symmetric voting systems aggregate information
well, see for example~\cite{Kalai:04} and the examples
provided in the current paper.

\subsection{Notation and Main Results}

Let $A$ be a finite set. Let $X = A^n$. For $\sigma \in S(n)$, a
permutation on $n$ elements and $x
\in A^n$ we denote by $y = x_{\sigma}$ the vector satisfying
$y_i = x_{\sigma(i)}$ for all $i$. For $\sigma \in S(A)$ we write
$y = \sigma(x)$ for the vector satisfying $y_i = \sigma(x_i)$ for all
$i$. For $a \in A$ and $x,y \in X$ we
write $x \leq_a y$ if $\{i : x_i = a\} \subset \{i : y_i = a\}$ and for all $i$ such that $y_i \neq a$ it
holds that $x_i = y_i$. In other words if $x \leq_a y$ then for all $i$ if $x_i \neq y_i$ then $y_i = a$.
It is easy to see that $\leq_a$ defines a partial order on $X$.

We say that $f : X = A^n \to A$ is {\em monotone} if
for all $a \in A$ and $x,y \in X$ such that $x \leq_a y$ it holds that
$f(x) = a$ implies that $f(y) = a$.
We say that $f$ is {\em symmetric} if there exists a transitive
group $\Sigma \subset S(n)$ such that $f(x_{\sigma}) = f(x)$ for all
$x \in X$ and $\sigma \in \Sigma$.
We say that $f$ is {\em fair} if for all $\sigma \in S(A)$ and all
$x \in X$ it holds that $f(\sigma(x)) = \sigma(f(x))$.

Let $\Delta[A]$ denote the simplex of probability measures on $A$ and
let $\gamma$ denote the standard probability measure on $\Delta[A]$.
For
$\mu \in \Delta[A]$ denote by $\P_{\mu}$ the measure
$\mu^{\otimes n}$ on $X$. We denote by $\E_{\mu}$ the expected value
according to the measure $\P_{\mu}$.
For any measure $\mu \in \Delta[A]$, we write $\mu^{\ast}$ for the minimal probability $\mu$ assigns to any of
the atoms in $A$.

In our main result we show that:

\begin{Theorem} \label{thm:main}
There exists an absolute constant $C = C(|A|)$ such
that if $f$ is symmetric and monotone then for
any $a \in A$ and $\eps > 0$ it holds that
\[
\gamma[\mu : \eps \leq P_{\mu}[f = a] \leq 1-\eps] \leq
C (\log(1-\eps) - \log(\eps)) \frac{\log \log n}{\log n}.
\]
\end{Theorem}
The result above establishes sharp threshold for symmetric functions
as it shows that for almost all probability measures $f$ takes one
specific value with probability at least $1-\eps$.
\begin{Remark}
It is interesting to compare the results established here to those
of~\cite{FriedgutKalai:96}. For $|A|=2$ our results give a threshold
interval of length $O(\log \log n/\log n)$ compared to the results
of~\cite{FriedgutKalai:96} which give threshold interval of length
$O(1/\log n)$. In the binary case the later result is tight.
It is natural to conjecture that the threshold is always of measure $O(1/\log n)$.
\end{Remark}

\subsubsection{Other applications}

Various applications of the threshold result
to problems involving distributions of edge colored graphs are given in
subsection~\ref{subsec:graph}.

\subsection {Social choice background, and applications}

We will describe two main applications of our main result to social choice theory. The first application gives an extension
of Condorcet's Jury theorem for monotone choice functions for more than two candidates and large classes of voting rules.
The second application is to indeterminacy results for generalized social choice functions.

\subsubsection{Aggregation of Information}
The law of large numbers implies that in an election between
two candidates denoted $0$ and $1$, if every voter votes for $0$
with probability $p>1/2$ and for $1$ with probability $1-p$ and if
these votes are independent, then as the number of voters
tends to infinity the probability that $0$ will be elected tends
to one. This fact is referred to as Condorcet's Jury Theorem.

This theorem can be interpreted as saying that
even if agents receive very poor but independent signals indicating
which  decision is correct,  majority voting will nevertheless
result in the correct decision being taken with a high
probability if  there are enough agents (and each agent votes
according to the  signal he receives).
This phenomenon is referred to
as asymptotically complete aggregation of information
and it plays an important role in theoretical economics.

More recent results studied aggregation of information for general
symmetric fair functions $f : \{0,1\}^n \to \{0,1\}$. Recall that
such a function is fair if $f(1-x_1,\ldots,1-x_n) =
1-f(x_1,\ldots,x_n)$. In this setup the results
of~\cite{FriedgutKalai:96} imply that for every $p > 1/2$ and every
symmetric fair function on $n$ voters it holds $\E_p[f] > 1 - o(1)$.

Here we derive a similar result in the case of an election between
$[q]$ candidates. Note that the conditions of monotonicity and fairness are both natural in this setup.
\begin{itemize}
\item
Monotonicity implies that if in a certain vector $x \in [q]^n$ of voters a certain candidate $i$ is elected and if $y$ is identical to $x$ except that some of the voters changed their mind as to vote $i$, then the outcome of the vote for $y$ should also be $i$.
\item
Fairness means that all the candidates are treated equally.
\end{itemize}

\begin{Theorem}
\label {t:condorcet}
For every $q$ there exists a constant $C = C(q)$ for which the following holds for every $\eps < 1/3$.
Let $\mu \in \Delta(q)$ and $i \in [q]$ satisfy that
\[
\mu(i) > \max_{j \neq i} \mu(j) + C (\log(1-\eps) - \log(1/q)) \frac{\log \log n}{\log n}.
\]
Then for every fair monotone function $f : [q]^n \to [q]$ it holds that
\[
\mu[f = i] \geq 1-\eps.
\]
\end{Theorem}

In words - the proposition claims that for any measure on the votes that has a $\log \log n / \log n$ bias towards one of the candidates $i$ and any fair monotone voting function it holds that $i$ will be elected with high probability. The proof is given in subsection~\ref{subsec:jury}.

\subsubsection{Indeterminacy}

Arrow's impossibility
theorem asserts that under certain natural conditions,
if there are at least three alternatives
then every non-dictatorial social choice
gives rise to a non-rational
choice function, i.e.,  there exist profiles
such that the social choice
is not rational. Arrow's theorem can be seen in the context of
Condorcet's ``paradox'' which demonstrates
that the majority rule may result in
the society preferring A over B , B over C and C over A.
Arrow's theorem shows that such ``paradoxes''
cannot be avoided with {\it any}
non-dictatorial voting method. It is the general form of Arrow's theorem,
which can be applied to general schemes for aggregating individual
rational choices, that made it so important in economic theory.

McGarvey~\cite{McGarvey:53} appears to have been the first to show that for every
asymmetric relation $R$ on a finite set of
candidates there is a strict-preferences
(linear orders, no ties) voter profile that has the relation $R$
as its strict
simple majority relation.
This implies that we cannot deduce
the society's
choice between two candidates even if we know the
society's choice between every other pair of candidates. We refer to this phenomena as ``complete indeterminacy''.

Saari~\cite{Saari:89} proved that the plurality method
gives rise to {\it every} choice function for sufficiently large societies.
(In fact, he proved more: the ranking on any subset of the
alternatives can be prescribed).
This implies that knowing the outcome of the plurality choice for
several examples, where
each example consists of a set $S$ of alternatives and the chosen
element $c(S)$ for $S$, cannot teach us
anything about the outcome for a set of alternatives which is
not among the examples we have already seen.

In \cite{Kalai:04} McGarvey's theorem is extended to sequences
of neutral social welfare functions
in which the maximum Shapley-Shubik power index tends to 0.
The proof relied on threshold properties of Boolean functions.
In particular,
McGarvey's theorem extends to neutral social welfare functions
which are invariant under a transitive group of permutations of the
individuals. We describe a similar extension of Saari's theorem.

\section{Proof of the main theorem} \label{sec:main}
The proof follows the same simple idea used in~\cite{FriedgutKalai:96}.
That is, we use information on the influences in order to deduce that $\E_{\mu}[f]$ varies quickly
as a function of $\mu$. We thus derive a generalization of Russo's formula~\cite{Margulis:74,Russo:81} which expresses derivatives in terms of
influences. We then use the fact that functions taking only a bounded number of values must have large influence sums.

Depending on the type of influence sum bounds, one obtain different results.
Using the results of~\cite{BKKKL:92} it is possible to obtain a bound on the threshold interval length of order
$O(\log^{1/3} n)$.Here we derive the better bound of order $O(\log \log n / \log n)$ using a generalization of the results of Talagrand~\cite{Talagrand:94}. This generalization was proven in a course taught by the second author on Fall 2005.
A draft of the proof was written in scribe notes of the course by Asaf Nachmias at\\
$http://www.stat.berkeley.edu/\sim mossel/teach/206af05/scribes/oct25.pdf$. For completeness we have corrected and completed the proof.

It seems like in order to obtain a threshold bound of $O(1 / \log n)$ it would be needed to derive a tighter influence sum bound tailored to the setup here.

From now on, without loss of generality assume $A = [q] = \{0,1\ldots,q-1\}$.
We will consider functions $f : [q]^n \to \{0,1\}$ and say that such
a function is {\em $0$-monotone} if $x \leq_0 y$ implies that
$f(x) \leq f(y)$.

Clearly in order to prove Theorem~\ref{thm:main},
it suffices to prove that there exists $C > 0$ such that if $f$ is symmetric and $0$-monotone then for all
$\eps > 0$ it holds that
\begin{equation} \label{eq:main_claim}
\gamma[\mu : \eps \leq P_{\mu}[f = 0] \leq 1-\eps] \leq C(\log(1-\eps) - \log(\eps)) \frac{\log \log n}{\log n}.
\end{equation}

\subsection{Influences}
We will use the $L_2$ {\em influence}

Let $f : X \to \R$.

Let $I_{\mu}^i(f)$, the $i$'th influence of $f$ according to $\mu \in
\Delta[A]$ by
\[
I_{\mu}^i(f) =
\E_{{\mu}}[\Var[f | x_{1},\ldots,x_{i-1},x_{i+1},\ldots,x_n]]
\]

In Corollary~\ref{cor:talag2}
we will derive the following lower bound on influence sum which is a consequence of a generalization of a result
of Talagrand~\cite{Talagrand:94}. We restate the corollary here:

\begin{Corollary} \label{talag}
There exists some universal constant $C$ such that for any probability space $(\Omega,\mu)$
and any function $f : \Omega^n \to \{0,1\}$ which is symmetric it holds that
it holds that
\[
\sum _{i=1}^n {I_i(f)} \geq \frac{C}{\log(1/\alpha)} \log n \Var(f).
\]
\end{Corollary}

\subsection{Russo Type Formula}

We denote by $T(\Delta)$ the tangent space to $\Sigma(\Delta)$. The
space is easily identified with the space of all vectors in $t \in
\R^A$ satisfying $\sum t_i = 0$. The natural derivative on
$T(\Delta)$ satisfies $\partial P_{\mu}[f] / \partial t \in T(\Delta)$.
\begin{Lemma} \label{lem:der1}
Let $f : [q] \to \{0,1\}$ be a $0$-monotone function and $\mu$ a measure
on $[q]$ and suppose that $f$ is $0$-monotone.
Write $\mu = (1-\mu(0)) \mu' + \mu(0) \delta_0$ where $\mu'(0) = 0$ and $\mu'$ is a probability measure.
Let $t = \delta_0 - \mu' \in T(\Delta)$. Then
\[
\frac{\partial \E_{\mu+ht}[f]}{\partial h} =
1(f \mbox{ is not constant}) \mu'[1-f] \geq \Var_{\mu'}[f].
\]
\end{Lemma}

\begin{proof}
By linearity the derivative equals $f(0) - \mu'[f]$.
The last expression is $0$ when $f$ is constant. If $f$ is not constant then by monotonicity it holds that
$f(0) = 1$ and therefore the derivative is $1-\mu'[f]$ as needed. Finally, the inequality follows from the fact that if $f$
is constant then the variance is $0$ and otherwise $\mu'[1-f] \geq \mu'[f] \mu'[1-f] = \Var_{\mu'}[f]$.
\end{proof}

We now prove a generalization of Russo's formula.

\begin{Lemma} \label{lem:der2}
Let $f : [q]^n \to \{0,1\}$ be a $0$-monotone function and $\mu \in
\Delta[q]$. Write $\mu = (1-\mu(0)) \mu' + \mu(0) \delta_0$ where $\mu'(0) = 0$ and $\mu'$ is a probability measure.
Let $t = \delta_0 - \mu' \in T(\Delta)$. Then:
\[
\frac{\partial \E_ {\mu+ht}[f]}{\partial h}_{| h = 0} =
\sum_{i=1}^n \E_{\mu}\left[ 1( [f | F_i] \mbox{ is not constant } )
\E_{\mu'}[1 - f| F_i] \right]
\geq \sum_{i=1}^n I_{\mu'}^i(f),
\]
where $F_i$ denotes conditioning on $x_1,\ldots,x_{i-1},x_{i+1},\ldots,x_n$.
\end{Lemma}

\begin{proof}
Use the chain rule:
\[
\frac{\partial \E_ {\mu+h t}[f]}{\partial h}_{| h = 0} =
\frac{\partial (\mu+h t)^{\otimes n}[f]}{\partial h}_{| h = 0} =
\sum_{i=1}^n \frac{\partial \otimes_{j=1}^n (\mu + h_j t) [f]}{\partial
  h_i}_{| (h_1,\ldots,h_n) = 0}
\]

and Lemma~\ref{lem:der1}.
\end{proof}

\subsection{Proof of the Main Result}

We now prove the main result, i.e., (\ref{eq:main_claim}). In the proof below the constants $c$ will denote
different constant at different lines. All of them depend on $q$ only.
\begin{proof}
Let $\Gamma$ denote the set of measures $\mu$ in $\Delta[q]$ satisfying the
following:
\begin{itemize}
\item $\mu[0] = 0.$
\item $\mu[i] \geq \eta$ for all $i \neq 0$.
\end{itemize}
Write $\mu^t = t \delta_0 + (1-t) \mu$ and note that
\begin{equation} \label{eq:biased_measures}
\gamma[ \mu^t : 0 \leq t \leq 1, \mu \notin \Gamma] \leq C(q) \eta.
\end{equation}
For each measure $\mu \in \Gamma$
we look at the interval $J_{\mu}^{\eps,1-\eps}$,
where
\[
J_{\mu}^{\eps,1-\eps} =
\{ t \in I_{\mu} : \eps \leq P_{\mu^t}[f] \leq 1-\eps\}.
\]
Our goal is to show that the interval $J_{\mu}^{\eps,1-\eps}$ is short by bounding the derivative of
\[
G(t) = \mu^t[f]
\]
in the interval.
In fact we will bound the length of each of the sub-intervals $J_{\mu}^{\eps,1/2}$ and $J_{\mu}^{1/2,1-\eps}$.
Since $f$ is $0,1$ valued, it follows that in the interval $J_{\mu}^{\eps,1/2}$
it holds that $\Var_{\mu^t}[f] \geq \frac{1}{2} \mu^t[f]$.
By Lemma~\ref{lem:der2} and Proposition \ref{talag} we conclude that:
\[
G'(t) \geq \sum_{i=1}^n I_{\mu^t}^i(f) \geq \frac{c \log n}{\log(1/\eta)} \Var_{\mu}[f] \geq
\frac{c \log n}{2 \log(1/\eta)} G(t).
\]
This implies that:
\[
(\ln G)'(t) \geq \frac{c \log n}{\log(1/\eta)}.
\]
In particular if $G(p) = \eps$ and $G(q) = 1/2$ then:
\[
\ln(1/2) - \ln(\eps) \geq \frac{c |p-q| \log n}{\log(1/\eta)}.
\]
So:
\[
|p-q| \leq \frac{(\ln(1/2) - \ln(\eps)) \log(1/\eta)}{c \log n}.
\]
Repeating the same argument for the interval $(1/2,1-\eps)$ we obtain that
\[
|J_{\mu}^{\eps,1-\eps}| \leq \frac{(\ln(1-\eps) - \ln(\eps)) \log(1/\eta)}{c \log n}.
\]
Which together~(\ref{eq:biased_measures}) with implies that
\[
\gamma[\mu : \eps \leq P_{\mu}[f = 0] \leq 1-\eps] \leq
C(q) \left( \eta + \frac{(\ln(1-\eps) - \ln(\eps)) \log(1/\eta)}{\log n} \right).
\]
Taking $\eta = (\ln(1-\eps) - \ln(\eps))/\log n$ we obtain the bound:
\[
C(q) \left( \frac{\ln(1-\eps) - \ln(\eps)}{\log n} +
  (\ln(1-\eps) - \ln(\eps)) \frac{\log \log n}{\log n} \right) \leq
 2 C(q) \left( (\ln(1-\eps) - \ln(\eps) \right) \frac{\log \log n}{\log n}
\]
as needed.

\end{proof}

\section{Applications}

\subsection{A general form of Condorcet's Jury theorem} \label{subsec:jury}

We begin with a proof of Theorem~\ref{t:condorcet}. Since the proof is similar to the previous proof, we only sketch the main steps.
\begin{proof}
We only sketch the proof since it is similar to the proof of Theorem~\ref{thm:main}.
Assume first that $\mu(j) > \frac{1}{\log n}$ for all $j$.
Let $\mu = t^{\ast} \delta_i + (1-t^{\ast}) \mu'$ where $\mu'$ is a probability measure with $\mu'(i) = 0$.
As before write $\mu^t = t \delta_i + (1-t) \mu'$ for every $t$.
Let $s < t^{\ast}$ be chosen so that $\mu^s(i) = \max_{j \neq i} \mu^s(j)$.
From symmetry and monotonicity it follows that
\[
\mu^s[f = i] \geq 1/q.
\]
Moreover, the argument in Theorem~\ref{thm:main} shows that for $t$ in the interval $(s,t^{\ast})$ it holds that
\[
\frac{\partial \mu^{t+h}[f = i]}{\partial h}_{| h = 0} \geq C \mu^t[f=i](1-\mu^t[f=i]) \frac{\log \log n}{\log n},
\]
which implies that
\[
|t^{\ast}-s| \leq C (\log(1-\eps) - \log(1/q)) \frac{\log \log n}{\log n},
\]
thus proving the statement of the theorem.

It remains to remove the assumption that $\mu(j) > \frac{1}{\log n}$ for all $j$.
 The general case where some of the probabilities $\mu(i)$ may satisfy $\mu(i) < \frac{1}{\log n}$ requires one additional step at the
 cost of taking the constant $C$ to be $C' = C+(q-1)$.
 Instead of the measure $\mu$ we first consider the measure $\hat{\mu}$ where
$\hat{\mu}(j) = \mu(j) + 1/\log n$ for $j \neq i$ and $\hat{\mu}(i) = \mu(i) - (q-1)/\log n$.

Note that if $\mu$ satisfies the conditions of the theorem with the constant $C'$ then $\hat{\mu}$ satisfies it
with the constant $C$. Moreover $\hat{\mu}(j) > 1/\log n$ for all $j$. Therefore by the first part of the proof
\[
\hat{\mu}[f = i] \geq 1-\eps.
\]
On the other hand, by monotonicity we have
\[
\mu[f = i] \geq \hat{\mu}[f = i],
\]
which implies the desired result.
\end{proof}

\subsection{Graph properties} \label{subsec:graph}
Consider a process on the edges of the complete graph $K_n$ where each
edge $e = \{a,b\}$ is labeled by $i$ with probability $\mu(i)$
for $1 \leq i \leq q$ independently for different edges.
This process defines sets $E_1,\ldots,E_q$
where $E_i$ is the set of edges labeled by $i$.
In other words, it defines $q$ graphs $([n],E_1),\ldots,([n],E_q)$.

\begin{Definition}

A function $A$ from the set of partitions of the edges of $K_n$ to $q$
parts into $[q]$ is called a {\em graph property} if:
For all partitions $E_1,\ldots,E_q$ and all permutations
$\sigma \in S(n)$ it holds that
\[
A(E_1,\ldots,E_q) = A(\sigma(E_1),\ldots,\sigma(E_q)),
\]
where $\sigma(E_i) = \{ \{\sigma(u),\sigma(v)\} : \{u,v\} \in E_i \}$.

The function $A$ is called a {\em monotone} graph property
if
\begin{itemize}
\item
For every pair of partitions
$(E_1,\ldots,E_q)$ and $(F_1,\ldots,F_q)$
\item
and all $1 \leq i \leq q$ it holds that
\end{itemize}
If:
\begin{itemize}
\item
$E_i \subseteq F_i$ and
\item
$F_j \subseteq E_j$ for $j \neq i$ and
\item
$A(E_1,\ldots,E_q) = i$,
\end{itemize}
Then $A(F_1,\ldots,F_q) = i$.

\end{Definition}

Here are a few examples:
\begin{itemize}
\item
Given a graph labeled by $[q]$, let $A(E_1,\ldots,E_q) = i$, where
$i$ is the minimal index for which $|E_i|$ is maximal. In other words,
$i$ is the most popular label, where ties are decided by preferring
the smaller index.
\item
Given a graph labeled by $[q]$ let $A(E_1,\ldots,E_q) = i$ where
$E_i$ has the largest clique. Again ties are decided by
preferring the smaller index.

\item
Given a graph labeled by $[q]$ let $A(E_1,\ldots,E_q) = i$ where
$E_i$ has the smallest independent set of vertices. Again ties are decided by
preferring the smaller index.

\end{itemize}

Theorem~\ref{thm:main} implies the following:

\begin{Corollary}
There exists a constant $C=C(q)$ such that for every monotone graph
property it holds that
\[
\gamma[\mu : \eps \leq P_{\mu}[A = i] \leq 1-\eps] \leq
C( \log(1-\eps) - \log(\eps)) \frac{ \log \log n}{\log n}.
\]
\end{Corollary}

\section {Indeterminacy for voting methods} \label{sec:int}

\label {s:cb2}

\subsection {The setting}

Let $P_+(X)$ denote the family of non-empty subsets of $X$.

\begin{Definition} Given a set $X$ of $m$ alternatives,
a {\em choice function} $c$ is a mapping which assigns to
each nonempty subset $S$ of $X$ an element $c(S) \in S$. 
A choice function is thus a map $c : P_+(X) \to X$ with the 
additional property that $c(S) \in S$ for all $S \in P_+(X)$.

A choice function is called {\em rational} if
there is a linear ordering on the alternatives such that $c(S)$
is the maximal element of $S$ according to that ordering.

A {\em social choice function} is a map

of the form $c=F(c_1,c_2,\dots,c_n)$
where $c$ is a choice function on $X$ which
depends on the profile of individual rational choice functions
$c_1, c_2, \dots c_n$ for the individuals.
\end{Definition}

Note that there are two
different meanings for the term ``social
choice functions''. Sometimes a social choice function
is referred to as a map which associates to the profile
of rational individual choices (or preferences) a single ``winner''
for the society. A social choice function in
this sense easily defines a social choice function
in our sense by restricting to a subset of alternatives.
We can regard social choice functions as election rules which given a
set $S$ of candidates and (strict) preference relations of the individuals
on the candidates, provides a rule for choosing the winner. We regard $X$
as the set of all possible candidates,
and $S$ as the set of available candidates, and we are interested
to understand the society's
choice as a function of $S$.

The axiom of Independence of Irrelevant Alternatives (Arrow's IIA)
asserts that
$c(A)$ may depend only on the preference relations restricted to the set $A$.

We require the stronger property Independence of Rejected
Alternatives (IRA), also referred to as Nash's IIA.

\begin{itemize}

\item[(IRA)] (Independence of Rejected Alternative (IRA))
$c(S)$ is a function of $(c_1(S), c_2(S), \dots c_n (S))$

Therefore we can write $c(S)=F_S(c_1(S),c_2(S),\dots c_m(S))$.

We will require a few more conditions:

\item[(P)] (Pareto)
$c(S) \in \{c_1(S),c_2(S),\dots,c_n(S)\}$.

\end{itemize}

\subsection {Further assumptions}

We will make the following additional assumptions:

\begin{itemize}
\item [(U)] (Unrestricted domains) The social
choice function is defined for arbitrary rational profiles
of the individuals.

\item[(N1')] (Neutrality)
The social choice is invariant under permutations of the alternatives.

\item[(A2')] (Weak anonymity)
The social choice is invariant under a transitive group of
permutations of the individuals.

\item[(M')] (Monotonicity) The
function $c(S)$  is monotone in the following sense: If $c(S)=s$ and
$c_i(S)=t$ for $t \ne s$ then changing the choice of the $i$-th individual
from $t$ to $s$ will not change $c(S)$.
\end {itemize}

\subsection {The result}

\begin {theo}
\label {t:mt2}
Let $c_0$ be an arbitrary choice function on $P_+(X)$ where $X$ is a set of $m$
alternatives and let $\delta < 1$. 

Then there is a probability distribution $\nu = \nu(c_0)$ on the space of orderings
of the alternatives and a number $N = N (\delta, m)$ 
such that the following holds  
for every social choice function $F$ 
which satisfies conditions (IRA) (P) (U') (N1'), (A2') and (M'):\\

If the number of individuals $n$ is larger than $N$ and if every
individual makes the choice randomly
and independently according to $\nu$, then with probability at least $\delta$ 
for all $S \in P_+(X)$ it holds that $c = F(c_1,...,c_n)$ satisfies $c(S)=c_0(S)$ with
probability  of at least $\delta$ (here $c_i(S)$ is the highest ranked alternative in the order of voter $i$). 
\end {theo}

\begin {cor}
\label {c:21}
There exists $N=N(m)$ such that when the number $n$ of
individuals is larger than $N(m)$, every choice function $c$ 
on $P_+(X)$, where $X$ is a set of $m$ alternatives, can be written as $c = F(c_1,\ldots,c_n)$ for 
every social choice function $F$ satisfies conditions (IRA), (P),
(U') (N1'), (A2') and (M') applied to $c_1,\ldots,c_n$ that are determined by rankings. 
\end {cor}

\subsection {The proof of Theorem \ref {t:mt2}} \label{subsection:inter}

\begin{proof}

Let $c_0$ be an arbitrary choice function and consider a profile with $n_0$ individuals
such that the plurality leads to $c_0$. Such a profile exists by
Saari's theorem. 

For an ordering $\pi$ of the alternatives
let $w'(\pi)$ be the number of appearances of the order $\pi$ and let
$w(\pi) = w'(\pi)/n_0$. 

Consider a random profile $(c_1,\ldots,c_n)$ on $n$ individuals
where for each individual $i$ the probability that $i$-th preference relation
$c_i$ is described by $\pi$ is $w(\pi)$ (independently for the individuals). 

We will show that for all $\delta' < 1$ if $n=n(\delta')$ is sufficiently large, then for $S \subset X$
it holds with probability
at least $\delta'$ that $c(S) = F_S(c_1(S),\ldots,c_n(S))$ satisfies $c(S) = c_0(S)$ with probability at least $\delta'$. This implies the required result by taking $1-\delta' \leq 2^{-m}(1-\delta)$.

To establish the claim above note that $c_1(S),\ldots,c_n(S)$ are i.i.d. 
and that for all $a \neq c_0(S)$ it holds that $P[c_i(S) = c_0(S)] \geq P[c_i(S) = a] + n_0^{-1}$. 
Furthermore the function $F_S$ as a functions of $c_1(S),\ldots, c_n(S)$ 
satisfies neutrality and weak anonymity. Thus from Theorem~\ref {t:condorcet} it follows that for $n$ large 
enough the probability that $F_S(c_1(S),\ldots,c_n(S)) = c_0(S)$ is at least $1-\delta'$ as needed. 

\end{proof}

Note that the proof actually implies the following 
in the spirit of a theorem by Dasgupta and Maskin \cite {DasguptaMaskin:08}. 

\begin {cor}
\label {c:22}
There exists $N = N (m)>0$ with the following property:
Let $U$ be a set of linear orders on $m$ alternatives.
If for the plurality rule there is
a profile restricted to  $U$ which leads to
a choice function $c$ for the society, then when $n>N(m)$ this is the case
for every social choice function
which satisfies conditions (IRA), (P), (N1'), (A2') and (M').
\end {cor}

\subsection {Taking rejected alternatives into account}

We now describe a crucial example suggested by Bezalel Peleg.
First, we note how we can base on choices
on pairs a choice correspondence: Given an asymmetric
binary relation $R$ on the set of
alternatives let $c(S)$ be the set of elements $y$ of $S$ such that
the number of $z \in S$ such that $yRz$ is maximal.
In other words, when we consider the directed graph described by the relation
we choose the vertex of maximal out-degree.

Let $\cal R$ denote the class of rational choice functions and
$\cal B$ denote the class of choice correspondences
obtained from binary relations $R$ as  just described.
Consider also the class $\cal B'$ of choice
functions obtained from $\cal B$ by
choosing a single element in $c(S)$ according to some fixed
order relation on the alternatives.
The number of choice functions in $\cal B'$ and the number
of choice correspondences in $\cal B$ is exponential in $n \choose 2$.

Now describe a
social choice function as follows: $aRb$ if a majority
of the society prefers $a$ to $b$.

In this case the social choice of $S$ does not depend solely on the individual
choices for $S$ but also on the preferences among pairs of elements in $S$.

When the individual choices are rational then
the social choice still belongs to the class
$\cal B$ (or $\cal B'$). In this case the choice from $S$ is simply
those elements of $S$ which are Condorcet winners
against the maximal number of other elements in $S$.
In this example the social choice for a
set $S$ is typically large but this apparently be corrected by
various methods of ``tie breaking''.

In these examples the size of the resulting classes
of choice functions  is exponential in a quadratic
function of $m^2$. It is much smaller than the number of all choice functions
which is double exponential in $m$.

The Borda rule can be analyzed by a similar consideration.
For this rule $c(A)$ is determined as follows:
For each alternative $a \in A$ let $r_i(a)$ be the number
of individuals who ranked $a$ in the $i$th place (among the elements of $A$).
Let $r(a) = \sum i \cdot r_i(a)$. The chosen element by the society $c(A)$
is the element of $a$ with the minimal weight.

Another way to describe the Borda rule is as follows:
First construct a
directed graph (with multiple edges)
with $A$ as the set of vertices by introducing an
edge from $a$ to $b$  for every individual that prefers $a$ to $b$.
Next, define (as before) $c(A)$ as the vertex with maximal outdegree.

It is easy to prove that the number of choice functions that arise in
this way is at most exponential in $m^3$. (The choice function can be
recovered from the sign patters of (less than) $2^m \cdot m^2$
linear expressions in $m^2$ real variables. 

To summarize,  the size of classes
of choice functions that arise
from a social choice function such that
$c(S)$ may depend on the individual preferences of the elements of $S$
is at least exponential in $m^2$ and this bound is sharp.

\section{A generalization of Talagrand's result}
In this section we prove the bounds on $L^2$ influence sums for arbitrary probability spaces.
For this we first recall the notion of Efron-Stein decomposition~\cite{EfronStein:81}
then generalize Talagrand's result~\cite{Talagrand:94} and finally derive corollaries for $\{0,1\}$-valued and symmetric functions.

\subsection{Efron-Stein decomposition}
Consider finite probability spaces $\Omega_1, \ldots, \Omega_n$,
with measures $\mu _1,\ldots , \mu_n$. Let $\alpha _i$ be size of
the smallest atom of $(\Omega_i, \mu_i)$, and set
$\alpha = \min_i \alpha _i$. Let $f \in L^2 ( \prod _i \mu _i)$ be a real valued
function. Write $f = \sum_{S \subset [n]} f_S$ for the Efron-Stein decomposition of $f$ which we now recall.

\begin{Definition} \label{def:efron_stein}
Let $(\Omega_1,\mu_1),\ldots,(\Omega_n,\mu_n)$ be discrete probability spaces
$(\Omega,\mu) = \prod_{i=1}^n (\Omega_i,\mu_i)$. The Efron-Stein decomposition of
$f : \Omega \to \R$ is given by
\begin{equation} \label{eq:efronstein}
f(x) = \sum_{S \subseteq [n]} f_S(x_S),
\end{equation}
where the functions $f_S$ satisfy:
\begin{itemize}
\item
$f_S$ depends only on $x_S$.
\item
For all $S \not \subseteq S'$ and all $x_{S'}$ it holds that:
\[
\E[f_S | X_{S'} = x_{S'}] = 0.
\]
\end{itemize}

\end{Definition}
It is well known that the Efron-Stein decomposition exists and
that it is unique~\cite{EfronStein:81}. We quickly recall the proof of existence.
The function $f_S$ is given by:
\[
f_S(x) = \sum_{S' \subseteq S} (-1)^{|S \setminus S'|} \E[f(X) | X_{S'} = x_{S'}]
\]
which implies
\[
\sum_{S} f_S(x) = \sum_{S'} \E[f | X_{S'} = x] \sum_{S : S' \subseteq S} (-1)^{|S \setminus S'|} =
\E[f | X_{[n]} = x_{[n]}] = f(x).
\]
Moreover, for $S \not \subseteq S'$ we have $\E[f_S | X_{S'} = x_{S'}] = \E[f_S | X_{S' \cap S} = x_{S' \cap S}]$
and for $S'$ that is a strict subset of $S$ we have:
\begin{eqnarray*}
\E[f_S | X_{S'} = x_{S'}] &=& \sum_{S'' \subset S} (-1)^{|S \setminus S''|} \E[f(X) |X_{S'' \cap S'} = x_{S'' \cap S'}] \\ &=& \sum_{S'' \subset S'} \E[f(X) |X_{S''} = x_{S''}] \sum_{S''  \subset \tilde{S} \subset S'' \cup (S \setminus S')} (-1)^{|S \setminus \tilde{S}|} = 0.
\end{eqnarray*}

\subsection{Generalization of a Result of Talagrand}
We now prove:
\begin{Theorem} [Generalization of Talagrand, 1994] \label{talagthm}
There exists some universal constant $C$ such that for any probability spaces
$(\Omega,\mu)$ and any function $f \in L^2(\Omega^n, \mu^n)$
it holds that
$$\Var(f) \leq C\log(1/\mu^{\ast}) \sum _{i\leq n} \frac {||\Delta _i
f||_2^2}{\log \Big (||\Delta_i f||_2 /||\Delta_i f||_1 \Big )} \,
.$$
where $\Delta _i f = \sum _{S : i\in S} f_S$.
\end{Theorem}
The proof is almost identical to Talagrand's proof using Efron-Stein decomposition instead of Fourier expansion and
known bounds on the hyper-contractive constants of finite probability spaces. In particular we'll use the following result of
Wolff~\cite{Wolff:07}:
\begin{Theorem} \label{thm:wolff}
For $g : \Omega^n \to \R$, let:
$$ T_\Theta g = \sum _S \Theta ^{|S|} g_S \, ,$$
Then for all $g$ it holds that:
\[
\| T_{\sigma} g \|_2 \leq \| g \|_{3/2},
\]
for all
$$ \sigma \leq \Big ( \frac{(1-\alpha)^{2-4/3} - \alpha^{2-4/3}}{(1-\alpha)\alpha^{1-4/3} - \alpha(1-\alpha)^{1-4/3}} \Big)^{1/2}.$$ . \end{Theorem}
\begin{Remark}
In particular one may take $\sigma = \alpha^2/6$.
\end{Remark}
\begin{proof}
Let
$$ f(x) = x^{2/3}, \quad g(x)=-(1-x)/x^{1/3}.$$
By Lagrange's theorem we may take:
$$ \sigma^2 = \frac{f'(\xi_1)}{g'(\xi_2)} \, ,$$
for some $\xi_1,\xi_2 \in (\alpha, 1-\alpha)$.
Clearly
\[
f'(x) = (2/3)x^{-1/3}, \quad g'(x) = (2/3)x^{-1/3} + (1/3)x^{-4/3},
\]
are decreasing, and therefore
\[
\sigma^2 \geq \frac{f'(1-\alpha)}{g'(\alpha)} =
\frac{2(1-\alpha)^{-1/3}}{2\alpha^{-1/3}+ \alpha^{-4/3}} \geq \frac{2}{3 \alpha^{-4/3}} \geq \frac{\alpha^2}{6}.
\]
\end{proof}

We now prove Theorem~\ref{talag}.
\begin{proof}
For a real function $g$
from our space, denote
$$ M^2(g) = \sum _{S \neq \emptyset} {\| g_S \|_2^2 \over |S|} \, .$$
So
$$\Var[f] = \sum _{S \neq \emptyset} \| f_S \|_2^2 = \sum _{i=1}^n M^2(\Delta _i f) \, .$$
Note that the statement of the theorem follows if we prove that for
any function $g$ with $\E g = 0$,
\begin{equation}\label{claim}
M^2(g) \leq K \log(1/\alpha) \frac{||g||_2^2}{\log\Big(
||g||_2/||g||_1 \Big )} \, .
\end{equation}
The statement of the theorem follows by apply (\ref{claim}) to $\Delta_i f$ and summing the inequalities.
To prove (\ref{claim}) we use hypercontractivity. The following
proposition is proved in the end of this note.

Applying Theorem~\ref{thm:wolff} gives that for any integer $k>0$,
$$ \sigma^{2k} \sum _{0 < |S|=k} \|g_S \|_2^2 \leq \sum _{S} \sigma^{2|S|} \| g_S \|_2^2  = ||T_\sigma g||_2^2
\leq ||g||_{3/2}^2 \, ,$$ hence
$$\sum _{|S|=k} \| g_S \|_2^2 \leq \Big ( {6 \over \alpha ^2} \Big
)^k ||g||_{3/2}^2 \, .$$

Fix an integer $m>0$, and sum the
previous inequality for all $k\leq m$ to get
$$ \sum _{|S| \leq m} {\| g_S \|_2^2 \over |S|} \leq \sum _{k\leq
m} {\Big ( {6 \over \alpha ^2} \Big )^k \over k} ||g||_{3/2}^2
 \leq {2\Big ( {6 \over \alpha ^2} \Big )^m \over m} ||g||_{3/2}^2 \, ,$$
 where the last inequality comes from the fact that the ratio
 between two consecutive summands in the sum is greater than $2$.
We now have
\begin{eqnarray} \label{final} M^2(g) &=& \sum _{|S|\leq m}
{\| g_S \|_2^2 \over |S|} + \sum _{|S|>m} {\| g_S \|_2^2 \over |S|}
\leq {2\Big ( {6 \over \alpha ^2} \Big )^m \over m} ||g||_{3/2}^2
+ {||g||_2^2 \over m} \nonumber \\ &\leq& {2 \over m}\Big[ \Big (
{6 \over \alpha ^2} \Big )^m  ||g||_{3/2}^2 + ||g||_2^2 \Big ]\,
.\end{eqnarray} We now choose optimal $m$. Choose largest $m$ such
that $\Big ( {6 \over \alpha ^2} \Big )^m ||g||_{3/2}^2 \leq
||g||_2^2$, hence
$$ \Big ( {6 \over \alpha ^2} \Big )^{m+1}  ||g||_{3/2}^2 \geq ||g||_2^2 \implies m+1 \geq
{2 \log \Big ( ||g||_2/||g||_{3/2} \Big ) \over \log (6/\alpha^2)}
\, .$$ Plugging this back into (\ref{final}) gives
$$ M^2(g) \leq C { \log(6/\alpha^2) ||g||_2^2 \over \log \Big ( ||g||_2/||g||_{3/2}
\Big ) } \, . $$ An application of Cauchy-Schwartz gives
$$ ||g||_{3/2}^3 \leq ||g||_1 ||g||_2^2 \, ,$$
hence
$$\Big ( {||g||_{3/2} \over ||g||_2 } \Big ) ^3 \leq { ||g||_1
\over ||g||_2 } \, ,$$ which concludes the proof of (\ref{claim})
and so we are done.
\end{proof}

\subsection{The formula for $\{0,1\}$ valued functions}
For $\{0,1\}$ valued functions, Theorem~\ref{talag} has a very simple formulation.
\begin{Lemma}
Let $f : \prod_{i=1}^n \Omega_i \to \{0,1\}$. Then
\[
\| \Delta_i f \|_1 = 2 I_i(f).
\]
\end{Lemma}

\begin{proof}
Let $f : [q] \to \{0,1\}$ be a function with $\E[f] = p$.
Note that
\[
\E[|f-p|] = p(1-p)+(1-p)p = 2 \Var[f].
\]
Since
\[
\Delta_i f = f - \E[f | X_1,\ldots,X_{i-1},X_{i+1},\ldots,X_n],
\]
we see that
\[
E[|\Delta_i(f)| | X_1,\ldots,X_{i-1},X_{i+1},\ldots,X_n] = 2 \Var[\Delta_i | X_1,\ldots,X_{i-1},X_{i+1},X_n],
\]
and therefore taking expected value of $X_1,\ldots,X_{i-1},X_{i+1},X_n$ we obtain that
\[
\| \Delta_i f \|_1 = 2 I_i(f).
\]
\end{proof}

We now obtain the following corollaries
\begin{Corollary} \label{cor:talag1}
There exists some universal constant $C$ such that for any probability spaces
$(\Omega,\mu)$,
and any function $f : \Omega^n \to \{0,1\}$
it holds that
$$\Var_{\mu}(f) \leq C\log(1/\mu^{\ast}) \sum _{i=1}^n \frac {I_i(f)}{\log(1/2) - \log(I_i)/2}
.$$
In particular if $I_i(f) \leq \delta$ for all $i$ then:
\[
\sum _{i=1}^n {I_i(f)} \geq \frac{1}{2 C \log(1/\mu^{\ast})} (\log(1/\delta) - \log(1/4)) \Var_{\mu}(f).
\]
\end{Corollary}

\begin{Corollary} \label{cor:talag2}
There exists some universal constant $C$ such that for any probability space $(\Omega,\mu)$
and any function $f : \Omega^n \to \{0,1\}$ which is symmetric it holds that
it holds that
\[
\sum _{i=1}^n {I_i(f)} \geq \frac{C}{\log(1/\alpha)} \log n \Var_{\mu}(f).
\]
\end{Corollary}

\begin{proof}
If $I_i \geq \frac{\log n}{100 n}$ the claim follows immediately.
Otherwise we apply the previous corollary and note that $\log(1/\delta) = \Omega(\log n)$.
\end{proof}

\begin{Remark}
It is a fundamental question both for the Boolean case and for the case of larger alphabet
that we consider here to find conditions (and even appropriate definitions) for sharp thresholds for monotone functions.
One may want to obtain similar result on the basis of the fact that
the function $f$ has low influences. However, this condition on its
own does not suffice.
Let $m : \{-1,1\}^{2n} \to \{-1,1\}$ be the
majority function defined as follows:
\[
m_n(x) = sign(\sum_{i=1}^{2n} x_i).
\]
When the sum is $0$ the function is defined arbitrarily in such a way
that it is $-1$ on half of the balanced inputs and $1$ on the other
half.
We write vectors $(x,y) \in \{-1,1\}^{2n}$ where $x$ and $y$ are two
vectors of length $n$. Let $f_n = m(x,-y)$.
Then clearly all of $f$ influences are of order at most $n^{-1/2}$. On the
other hand, it is easy to see that $\lim_{n \to \infty}
\E_{p(n)}[f] = 0$ for all functions $p(n)$ satisfying
$\lim_{n \to \infty} n p(n) = 0$ and $\lim_{n \to \infty} n (1-p(n)) =
0$.
\end{Remark}

\bibliographystyle{abbrv}
\bibliography{all}

\end {document}